\documentclass[11pt,letterpaper]{amsart}

\usepackage[english]{babel}
\usepackage[utf8x]{inputenc}
\usepackage[T1]{fontenc}
\usepackage[normalem]{ulem}

\usepackage{amsmath}
\usepackage{amsfonts}
\usepackage{amsthm}
\usepackage{graphicx}
\usepackage{enumerate}
\usepackage[colorinlistoftodos]{todonotes}
\usepackage[colorlinks=true, allcolors=blue]{hyperref}

\numberwithin{equation}{section}
\newtheorem{theorem}{Theorem}[section]
\newtheorem{corollary}[theorem]{Corollary}
\newtheorem{lemma}[theorem]{Lemma}
\newtheorem{proposition}[theorem]{Proposition}

\newtheorem{claim}[theorem]{Claim}
\theoremstyle{definition}

\theoremstyle{remark}
\newtheorem{remark}[theorem]{Remark}
\newtheorem{example}[theorem]{Example}

\newcommand{\R}{\mathbb{R}}
\newcommand{\Z}{\mathbb{Z}}

\newcommand{\conv}{\operatorname{conv}}
\newcommand{\sgn}{\operatorname{sign}}
\newcommand{\ext}{\operatorname{ext}}

\newcommand{\colvec}[2]{\left(#1, #2 \right)^T}

\title{Strictly Expansive Matrices}
\author{Darrin Speegle}
\address{Department of Mathematics and Statistics, Saint Louis University, St. Louis, MO 63103, USA}
\email{speegled@slu.edu}

\begin{document}

\begin{abstract}
The class of strictly expansive matrices was introduced in \cite{BowSpe02MeyerType}, where it was shown that if $A$ is an integer valued, strictly expansive matrix, then there exists an orthonormal $A$-wavelet whose Fourier transform is compactly supported and smooth. We show that strongly connected diagonally dominant integer matrices are strictly expansive, and that integer matrices with determinant two are not strictly expansive with respect to particularly nice sets.
\end{abstract}

\maketitle

\section{Introduction}

An $n \times n$ integer matrix $A$ is said to be \textit{expansive} if all of the eigenvalues of $A$ are larger than one in modulus. The expansive matrix $A$ is said to be \textit{strictly expansive relative to $K$} if $K$ is compact,  
\[
\sum_{k \in \Z^n} I_K(x + k) = 1\,\, a.e., 
\]
and $A(K^\circ) \supset K$. An expansive, integer matrix $A$ for which there exists such a $K$ is said to be \textit{strictly expansive}. 

In this paper, we study matrices that are strictly expansive, as well as matrices that cannot be strictly expansive with respect to a particularly nice set. The two main results are first, if the integer matrix $A$ is strongly connected, expansive and strictly diagonally dominant, then $A$ is strictly expansive. This can be seen as a complement to the widely cited result of Varah \cite{Var75}, which states that if for every $i$ $\left|a_{ii}\right| \ge \sum_{j \not= i} \left|a_{ij}\right | + \alpha$, then $\|A^{-1}\|_\infty \le 1/\alpha$. The second result is that if $\left|\det(A)\right | = 2$, then $A$ is not strictly expansive relative to a unit ball of a Banach space. In fact, more is shown. We show that such $A$ is not strictly expansive relative to a set $K$ which is symmetric, nor to a set $K$ which is convex, nor to a set $K$ with connected boundary. Nonetheless, it remains open whether the matrix $A = \begin{pmatrix}
    0&1\\2&0
\end{pmatrix}$ is strictly expansive.

Our reasons for studying strictly expansive matrices are many. First, an orthonormal $A$ (multi)-wavelet $(\psi^1, \ldots, \psi^M\}$ is said to be a \textit{Meyer-type wavelet} if each $\hat \psi^i$ is smooth and compactly supported. It is an open question in dimension $n \ge 3$ whether for every expansive $n \times n$ matrix with integer entries Meyer-type wavelets exist with the minimal $M = \left|\det(A)\right| - 1$. However, if $A$ is a strictly expansive matrix, then there exists a Meyer-type orthonormal $A$ wavelet with $|\det(A)| - 1$ generators \cite{BowSpe02MeyerType}. While this is not known to be a necessary condition for the existence of Meyer-type wavelets, it is a relatively simple condition which guarantees the existence of nice wavelets. Additionally, the case $n = 2$ was solved by showing that there are only seven equivalence classes of expansive, integer valued, $2 \times 2$ matrices which are not strictly expansive. One goal of this paper is to find criteria by which matrices can easily be classified as strictly expansive. Another goal of this program is to determine whether the results on strictly expansive matrices in $n = 2$ generalize to all $n$ in a meaningful way.

Second, if $A$ is strictly expansive and $S$ is an integer matrix with determinant $\pm 1$, then $S A S^{-1}$ is a strictly expansive matrix. Given an integer matrix $A$, it can be challenging to understand its equivalence class via conjugation by matrices in $GL_n(\Z)$ \cite{EHO2019, Gru1980}. When $n = 2$, however, the situation is much more straightforward and there are techniques going back to \cite{LatMac33} for finding ``nice" representatives of matrices. Indeed, finding these representatives was crucial in proving that, when $n = 2$, if the determinant of the integer valued matrix $A$ is large enough in absolute value, and all of the eigenvalues of $A$ have modulus larger than 1, then $A$ is strictly expansive in \cite{BowSpe02MeyerType}. Determining whether a matrix is strictly expansive can be seen as a test problem for overcoming the obstacles presented by the lack of understanding of nice representatives via conjugation by matrices in $GL_n(\Z)$. 

Finally, the interplay of tilings via translations and other actions on $\R^n$ is an active area of research. Finding sets which simultaneously tile by both integer translations and $A$ dilations for various $A$ was studied in \cite{BowSpe25, DaiDiaoGuHan02, DaiLarSpe97}. Finding sets which simultaneously tile by translations along two different lattices was studied in \cite{HanWang04}. The Steinhaus problem asks when there is a set which tiles by translations along all rotations of $\Z^n$, see for example \cite{JM, KP2017, KP2022}. The work most closely related to the problem studied in this paper is \cite{GuHan2000} and \cite{BowRzeSpe2001}. In \cite{GuHan2000}, it was shown that if $A$ is an expansive, integer matrix with determinant 2, then there exists a measurable set $K$ that tiles $\R^n$ by integer translations and whose interior contains the origin such that $AK \supset K$. In \cite{BowRzeSpe2001}, this result was extended to all expansive, integer matrices. In Proposition \ref{prop:compactmra} we provide a modest improvement of these results by showing that $K$ can be chosen to also be compact. We note also that if $A$ is strictly expansive with respect to $K$, then $K$ tiles via integer translations, and $AK \setminus K$ tiles via integer dilations by $A$, see Proposition \ref{prop:tiling}.

\section{Results in \texorpdfstring{$\R^n$}{n dimensions}}

Let $K\subset \R^n$ be a compact, convex, centrally symmetric set. For $x \in \R^n$, we write $\|x\|_K$ as shorthand for the norm of $x$ considered as an element of $\R^n$ with norm whose unit ball is $K$. For $A$ an $n\times n$ matrix, we write $\|A\|_K$ as shorthand for the norm of $A$ considered as a linear operator from $(\R^n, \| \|_K) \to (\R^n, \| \|_K)$.

We say that a countable collection $\{F_i: i \in I\}$ of measurable subsets of $\R^n$ \textit{packs} $\R^n$ if 
\[
m(F_i \cap F_j) = 0
\]
whenever $i \not= j$, where $m$ is Lebesgue measure. We say the collection \textit{covers} $\R^n$ if $\cup_{i \in I} F_i = \R^n$ up to a set of measure zero. The collection \textit{tiles} $\R^n$ if it both packs and covers.

We organize some useful and related facts here.

\begin{proposition}\label{prop:tiling} Let $A$ be an $n\times n$ real (not necessarily integer) matrix.
\begin{enumerate}
\item\label{prop:tiling:a} If $A$ is expansive, then there exists an inner product on $\R^n$ such that $A(B^\circ) \supset B$, where $B = \{x\in\R^n: \left|\langle x, x \rangle\right| \le 1\}$, see e.g. Lemma 1.5.1 in \cite{Szlenk84}. 
\item\label{prop:tiling:b} If $K$ is a compact, convex, centrally symmetric set, then $A(K^\circ) \supset K$ if and only if $\|A^{-1}\|_K < 1$.
\item\label{prop:tiling:c} If $A$ is expansive, then there exists a norm on $\R^n$ under which $\R^n$ is not a Hilbert space and such that $\|A^{-1}\| < 1$.
\item\label{prop:tiling:d} If $A$ is strictly expansive and $S$ is an integer matrix with determinant $\pm 1$, then $SAS^{-1}$ is strictly expansive. 
\item\label{prop:tiling:one} If $A$ is strictly expansive relative to $K$, then $K$ contains a neighborhood of the origin.
\item\label{prop:tiling:two} If $A$ is strictly expansive relative to $K$, then $AK \setminus K$ tiles $\R^n$ via dilations by $A$.
\end{enumerate}
\end{proposition}

\begin{proof}
    Item (\ref{prop:tiling:b}) is a restatement of the definition of $\| \cdot \|_K$. Item (\ref{prop:tiling:c}) follows from item (\ref{prop:tiling:a}) and compactness. For (\ref{prop:tiling:d}), use that if $K$ tiles by $\Z^n$ translations and $B$ is an invertible, integer matrix, then $BK$ tiles by $B\Z^n$ translations. 

       For (\ref{prop:tiling:one}), note that $A^{-1}(K) \subset K^\circ \subset K$, and by repeating this, $A^{-j}(K) \subset K$ for all $j \ge 1$. Since $A$ is expansive, for every $x \in \R^n$, $A^{-j}x \to 0$ as $j \to \infty.$ Therefore, $0 = A(0) \in K$, so $0 \in K^\circ$.

   For (\ref{prop:tiling:two}), let $W = AK \setminus K$. If $j < k$, then $A^{j + 1}K \subset A^k K$, and 
   \[
   A^jW \cap A^kW \subset A^{j + 1}K \cap (A^{k + 1} K \setminus A^k K) = \emptyset,
   \]
    so $\{A^jW\}$ packs $\R^n$.  To see that $\{A^jW\}$ covers $\R^n$, note that
    \[
    \bigcup_{j \in \Z} A^jW  = \bigcup_{j \in \Z} A^j K \setminus \bigcap_{j \in \Z} A^{j} K = \R^n \setminus\{0\}
    \]
    with the last equality following because $A$ is expansive and part (\ref{prop:tiling:one}).
\end{proof}

Of particular interest is when $A$ is strictly expansive relative to $K = [-1/2, 1/2]^n$. From (\ref{prop:tiling:b}), $\|A^{-1}\|_\infty < 1$ if and only if $A$ is strictly expansive with respect to $K$, where $\|A\|_\infty$ is the norm of $A$ considered as an operator from $\ell_\infty^n\to \ell_\infty^n$. Estimating $\|A^{-1}\|_\infty$ based on various properties of the matrix $A$, has been considered for diagonally dominant matrices and their generalizations in, for example, \cite{CDDL2013, SWJYM1996, Var75}. We mention in particular the following theorem.

\begin{theorem}[Varah \cite{Var75}]\label{Var75}
Let $A$ be a row diagonally dominant matrix (not necessarily integer valued). Let 
\[
\alpha = \min_j \left(\left|a_{jj}\right| - \sum_{k \not= j} \left|a_{jk}\right|\right)
\]
Then, $\|A^{-1}\|_\infty \le 1/\alpha$.
\end{theorem}

\begin{corollary}
    If the columns of the integer matrix $A$ can be re-arranged so that $A$ is diagonally dominant with constant $\alpha = 2$, then $A$ is strictly expansive with respect to the unit edge length cube. 
\end{corollary}

\begin{proof}
    Let $P$ be a permutation matrix such that $AP$ is diagonally dominant with constant $\alpha = 2$. Then,  $\|(AP)^{-1}\| \le 1/2 < 1$, so $AP$ is strictly expansive with respect to the unit cube. 

    Since $PK = K$ for $K$ the unit cube, we are done. 
\end{proof}

Next we wish to show that strongly connected, expansive, diagonally dominant integer matrices with $\alpha = 1$ are strictly expansive. We need a few lemmas, which we prove for completeness. See, e.g., \cite{McMullen1980} for similar results.

%Stopped here June 3 LAST READ

\begin{lemma}\label{expoint}
Let $K_1$ and $K_2$ be compact, convex subsets of $\R^n$ with $K_1^\circ \cap K_2^\circ = \emptyset$. Let $x_0$ be an extreme point of $K_1$. Then, for every $\epsilon > 0$, $B_\epsilon(x_0) \not\subset K_1 \cup K_2$. 
\end{lemma}

\begin{proof}
Note that $x_0 \in \partial K_0$. Since $K_1^\circ \cap K_2^\circ = \emptyset$, $x_0 \not\in K_2^\circ$. If $x_0 \not \in K_2$, choose $\epsilon > 0$ such that $B_\epsilon(x_0) \cap K_2 = \emptyset$. 

The last case to consider is $x_0 \in \partial K_2$. Since $x_0$ is an extreme point, there is a linear functional $x^*$ and a real number $c$ such that $K_1 \subset \{x: x^*(x) \le c\}$ and $x^*(x) < c$ for all $x \in K_1^\circ$, while $K_2 \subset \{x:x^*(x) \ge c\}$. 

If $K_1 \cup K_2 \supset B_\epsilon(x_0)$, then $K_1 \supset B_\epsilon(x_0) \cap \{x:x^*(x) < c\}$. However, since $K_1$ is closed and $x^*(x_0) = c$, this means that $K_1 \supset B_{\epsilon}(x_0) \cap \{x:x^*(x) \le c\}$, which contradicts $x_0$ being an extreme point of $K_1$.
\end{proof}

\begin{lemma}\label{boundary}
Let $K$ be a compact set which tiles by translations. 
\begin{enumerate}
\item For each $x\in \partial K$ there exists $y\in \partial K$ such that $x - y \in \Z^n\setminus \{0\}$.
\item If in addition, $K$ is convex, then for each $x\in \ext(K)$, there exists distinct $y_1, y_2 \in \partial K$ such that $x - y_i \in \Z^n \setminus\{0\}$ for $i = 1,2$.
\end{enumerate} 
\end{lemma}

\begin{proof}
For the first part, let $x \in \partial K$. We wish to show there is a $0\not= k \in \Z^n$ such that $x \in K + k$.  To that end, note that there exists a finite set $k_1, \ldots, k_m \in \Z^n \setminus \{0\}$ such that $d(x, K + k_i) < 1$. For each $1 \le i \le m$, if $x \not\in K + k_i$ then $d(x, K + k_i) > 0$. Choose $\epsilon < \min_{1\le i \le m} d(x, K + k_i)$ and see that 
\[B_\epsilon(x) \not\subset \cup_{k \in \Z^n} K + k,\]
a contradiction. 

Therefore, for each $x \in \partial K$, there exists $y \in K$ and $0 \not= k = k_y$ such that $x = y + k$.  It is clear, moreover, that $y\not\in K^\circ$, so $y \in \partial K$.

For the second statement, we have already shown that there exists $y_1$ with $x = y_1 + k$, where $k \in \Z^n$. By Lemma \ref{expoint}, $B_1(x) \not \subset K \cup K + k$. There are finitely many $k_1, \ldots, k_m$ not equal to $k$ such that $K + k_i \cap B_1(x) \not= \emptyset$. If none of $K + k_i$ contain $x$, then let $\epsilon < \min\{d(x, K + k_i)\}$ and see that $B_\epsilon(x) \not\subset \cup_{k \in \Z^n} (K + k)$, which contradicts $K$ tiling by translations. Therefore, for at least one $i$, $x \in K + k_i$. Therefore, $x = y_2 + k_i$, where $y_2 \not= y_1$ and $y_2 \in \partial K$ by necessity.
 \end{proof}

\begin{lemma}\label{surgery}
Let $A$ be an expansive $n\times n$ matrix. Let $K\subset \R^n$ be a compact, convex, centrally symmetric set that tiles by $\Z^n$ translations.  Suppose that there are exactly two points $x, y \in K$ such that $\|A^{-1}x\|_K = 1$ and $\|A^{-1} y \|_K = 1$. Then, $A$ is strictly expansive.
\end{lemma}

\begin{proof}

%We start by showing that $\|A^{-1}\|_K = 1$. Since $\|A^{-1}x\|_K = 1$ and $\|x\|_K \le 1$, $\|A^{-1}\|_K \ge 1$. If there is $z \in K$ such that $\|A^{-1}z\|_K > 1$, then let $z_0 = \frac{z}{\|A^{-1}z\|_K} \in K^\circ$ and see that $\|A^{-1} \|_K = 1$. Let $U$ be an open set contained in $K$. Then $A^{-1}U$ is an open set containing $A^{-1}z_0$, which has more than 2 points with norm 1, a contradiction. Therefore, $\|A^{-1}\|_K = 1$.

Note that the two points $x, y$ in the lemma statement must be antipodal. We re-write the two points as $x_0$ and $-x_0$. In addition, since $1 = \|A^{-1}\|_K = \max\{\|A^{-1}x\|_K: x \in \ext(K)\}$, $x_0$ and $-x_0$ are extreme points of $K$. By Lemma \ref{boundary}, there exists $y_1 \not= y_2$ on the boundary of $K$ such that $y_i= w + k_i$. At least one of the $y_i$ is not equal to $-x_0$; we choose that one and call it $y$ and write $y = x_0 + k$.

Note that $x_0 \not\in A^{-1}K$. Indeed, if $z \in K$ and $A^{-1}z = x_0$, then by assumption $z = \pm x_0$ and $A$ is not expansive.

Let $\epsilon_1 = \frac 12 d(A^{-1}K, x_0) > 0$. Let $\delta = \min\{\frac 12 d(A^{-1}y, \partial K), \epsilon_1\} > 0$ and choose $\epsilon_2 > 0$ such that $A^{-1}\overline{B_{\epsilon_2}(y)} \subset B_\delta(A^{-1} y)$. Let $\epsilon$ be the smaller of $\epsilon_1$ and $\epsilon_2$ and let $S = \overline{B_\epsilon(x_0)} \cap K$. Let 
\[
T = \overline{\left(K \setminus (S\cup -S)\right)} \cup (S + k) \cup (-S - k) 
\]
and see that $T$ tiles by $\Z^n$ translations. 

It remains to see that $A^{-1}(T) \subset T^\circ$. First, note that $A^{-1} \left(\overline{K \setminus (S \cup -S)} \right) \subset K^\circ$ and $S \cap A^{-1}K = \emptyset$ imply that 
\[
A^{-1}\left(\overline{K \setminus (S \cup -S)}\right) \subset K^\circ \setminus (S \cup -S) \subset T^\circ.
\]
To conclude, we need to see that $A^{-1}(S + k) \subset K^\circ \setminus (S \cup -S)$. We have $A^{-1}(S + k) \subset B_\delta(A^{-1} y) \subset K^\circ$. We also need that $A^{-1}(S + k) \cap \pm S = \emptyset$. Indeed, if $z \in A^{-1}(S + k) \cap S$, then 
\begin{align*}
d(x_0, A^{-1}K) &\le d(x_0, A^{-1}y) \\
& \le d(x_0, z) + d(A^{-1}K, z) \\
&< \epsilon + \delta \le 2\epsilon \le d(x_0, A^{-1}K),
\end{align*}
a contradiction. Similar arguments hold for $-S - k$ and $-S$, and the proof is complete.
\end{proof}

In some contexts, the previous lemma is more easily verified for $A$ rather than for $A^{-1}$. This is possible via the following lemma.

\begin{lemma}\label{positivesurgery}
Let $A$ be an expansive matrix, and let $K$ be a  compact, convex, centrally symmetric set which tiles by $\Z^n$ translations. Suppose that there are exactly two points $x, y \in \partial K$ such that $\|Ax\|_K = \|Ay\|_K = 1$. Then $A$ is strictly expansive.
\end{lemma}

\begin{proof}
It is clear that $Ax$ and $Ay$ are in the boundary of $K$, as are $A^{-1}Ax$ and $A^{-1}Ay$. If there is a third point $z$ in the boundary of $K$ such that $A^{-1}z \in \partial K$, then $A^{-1}z$ is a third point in $\partial K$ for which $A(A^{-1}z) \in \partial K$, which contradicts the lemma statement. Therefore, the result follows from Lemma \ref{surgery}. 
\end{proof}

Let $A = (a_{ij})$ be an $n\times n$ matrix. We associate a directed graph $G = (V, E)$ to $A$ in the following way. The vertex set is given by $V = \{1, \ldots, n\}$. The edge $(i, j)$ is in $E$ if $a_{ij} \not= 0$.

\begin{theorem}\label{thm:connected}
  Let $A$ be an expansive, integer valued, diagonally dominant matrix whose associated graph is strongly connected. Then, $A$ is strictly expansive. 
\end{theorem}

\begin{proof}
  Let $F$ denote the extreme points of the unit ball $B$ in $\ell_\infty^n$.  Note that $\|A^{-1}\|_\infty = \sup_{x\in F} \|A^{-1}x\|_\infty = \sup_{Ax \in  F} \|x\|_\infty$. Also, note that Theorem \ref{Var75} implies that $\|A^{-1}\|_\infty \le 1$. If the norm is strictly less than 1, we are done. What remains is the case that the norm is exactly equal to 1. 
  
By Lemma \ref{positivesurgery}, it suffices to show that there are exactly two norm one vectors in $B$ such that $\|Ax\|_\infty = 1$. Let $x\in B$ with norm 1 be such that $\|Ax\|_\infty = 1$. 

\begin{claim}
    If $|x_m| = 1$ and $(m, j) \in E$, $x_j = -\sgn(x_m a_{mm} a_{mj})$.
    
\end{claim}

\begin{proof}[Proof of claim]
    We have
    \begin{align*}
        1 & \ge |\langle Ax, e_m\rangle | \\
        &\ge |x_m a_{mm}| - \sum_{j \not= m} |x_j a_{mj}| \ge 1
    \end{align*}
    with equality if and only if $x_j = -\sgn(x_m a_{mm} a_{mj})$ whenever $a_{mj} \not = 0$ (and $|x_m a_{mm}| - \sum_{j \not= m} |x_j a_{mj}| = 1$).
\end{proof}

Continuing the proof of the theorem, there exists some $m_0$ such that $|x_{m_0}| = 1$. By the above claim and the fact that the graph associated with $A$ is connected, each $x_j$ is determined up to the sign of $x_{m_0}$. That is, there are exactly two vectors in $B$ for which $\|Ax\| = 1$, which completes the proof. 
 \end{proof}

\begin{remark}
    We point out that the proof of Theorem \ref{thm:connected} implies that many integer matrices with $\alpha = 1$ in Theorem \ref{Var75} have $\|A^{-1}\|_\infty < 1$.
\end{remark}

Now, we turn to negative results. The author does not know of any integer valued expansive matrix which is not strictly expansive. However, it is natural to ask whether $K$ can be chosen to have additional, nice properties such as convexity or symmetry (or both). In the remainder of this section, we investigate instances where this is impossible.

We start with some elementary, but useful observations.

\begin{lemma}\label{nointersect}
Let $K\subset \R^n$ be a compact set and $\Gamma \subset \R^n$ be a full rank lattice. If $\left(K - K\right) \cap \Gamma \not= \{0\}$. Then no open set containing $K$ packs by $\Gamma$ translations.
\end{lemma}

\begin{proof}
    Suppose $x, y \in  K$ are such that $x - y \in \Gamma \setminus \{0\}$. If $B_\epsilon(x)$ and $B_\epsilon(y)$ are subsets of $K$, then 
    \[
    \sum_{\gamma \in \Gamma} I_{B_{\epsilon}(x)} (\cdot  + \gamma) \ge 2
    \]
    and $K$ does not pack by $\Gamma$ translations.
\end{proof}

\begin{lemma}
If $K$ is a compact set that tiles $\R^n$ by translations, then 
\[
\partial K = \{x \in K: \tau_K(x) \ge 2\}
\]
and 
\[
K^\circ = \{x\in K: \tau_K(x) = 1\}
\]
\end{lemma}

\begin{proof}
    The first statement is a restatement of Lemma \ref{boundary}, and the second statement is a restatement of Lemma \ref{nointersect}.
\end{proof}

The following lemma will be applied to points in the boundary of $K$, but is valid for all points in $K$.

\begin{lemma}\label{threepoints}
Let $K\subset \R^n$ be compact, and let $\Gamma \subset \Z^n$ have index two in $\Z^n$. If there exist distinct points $x, y, z \in K$ such that $x - y \in \Z^n$ and $x - z\in \Z^n$, then no open set which contains $K$ can pack by $\Gamma$ translations.
\end{lemma}

\begin{proof}
By Lemma \ref{nointersect}, it suffices to show that there exists distinct points $x, y \in K$ such that $0\not= x - y \in \Gamma$. 

If $k_1 = x - y \in \Z^n \setminus \Gamma$ and $k_2 = x - z \in \Z^n \setminus \Gamma$, then $y - z = k_2 - k_1 \in \Gamma$ since $\Gamma$ is index 2, and we are done.
\end{proof}

\begin{lemma}
Let $K$ be a compact set that tiles $\R^n$ by translations. Then
\[
\tau_K(x) = \sum_{k \in \Z^n} I_K(x + k) \ge 1
\]
for \textbf{all} $x \in \R^n$.
\end{lemma}

\begin{proof}
Suppose $\tau_K(x) = 0$. Let $k_1, \ldots, k_L$ denote the finitely many integers such that $K + k_i$ is within 1 of $x$. For $0 < \epsilon < 1$, 
\[
B_\epsilon(x) \cap \bigcup_{k \in \Z^n} (K + k) = B_\epsilon(x) \cap \bigcup_{i = 1}^L (K + k_i).
\] 
Since $\bigcup_{i = 1}^L (K + k_i)$ is compact, if $x \not\in  \bigcup_{i = 1}^L (K + k_i)$, then there exists $0 < \epsilon < 1$ such that $B_\epsilon(x) \cap \bigcup_{i = 1}^L (K + k_i) = \emptyset$. It follows $\tau_K(x) = 0$ for all $x \in B_\epsilon(x)$, contradicting the assumption of the lemma.
\end{proof}

\begin{theorem}\label{index2}
Let $\Gamma \subset \Z^n$ be a full-rank lattice with index 2. Let $K \subset \R^n$ be a compact set which tiles by $\Z^n$ translations. If either
\begin{enumerate}
\item $K$ is symmetric, or
\item $K$ is convex, or
\item $\partial K$ is connected
\end{enumerate}
then no open set which contains $K$ packs $\R^n$ by $\Gamma$ translations. 
\end{theorem}

\begin{proof}
Suppose $K$ is symmetric and assume $n > 1$. Let $H$ be an integer matrix in column Hermite normal form such that $\det(H) = 2$ and $H\Z^n = \Gamma$, so that $\Gamma$ is generated by the columns of $H$. At most one column is not equal to the corresponding unit vector basis element, so we can conjugate by permutations so that $h_{11} = 2$. If there are any off-diagonal positive values in $H$, they must be ones and in the first column. By another permutation we can assume $h_{i1} = 1$ for $i \in I$, where $I \subset \{2, \ldots, n\}$ is possibly empty. Let $x = \frac 12 e_2$, where $\{e_1, \ldots, e_n\}$ is the standard unit vector basis for $\R^n$.  There exists $k \in \Z^n$ such that $x + k$ and $-x - k$ are in $K$, and $0 \not= 2x + 2k \in \left(K - K\right).$ To finish the proof, we need to show that $2x + 2k \in \Gamma \setminus \{0\}$, and the result will follow from Lemma \ref{nointersect}.

Indeed, write $2x + 2k = (v_1, \ldots, v_n)$ and see that $v_2 \not= 0$ and
\[
  2x + 2k =  k_1 h_1 +  \sum_{i \in I} \left(v_i - k_1\right) h_i + \sum_{i \in ([2, n] \setminus I)} v_i h_i  
  \]
 as desired.

For the second item, Minkowski showed that if a convex body tiles $\R^n$ via translations, then the body is necessarily centrally symmetric, see e.g. Lemma 2 in \cite{McMullen1980} for a proof. For a direct proof, suppose $K$ is compact and convex. By Lemma \ref{boundary}, there exist distinct $y_1, y_2, y_3 \in \partial K$ with $y_i - y_j \in \Gamma$. The result then follows from Lemma \ref{threepoints}.

For the third item, for each $0\not= k \in \Z^n$, let $B_k = \{x: x \in \partial K, x + k \in \partial K\}$. By Lemma \ref{boundary}, each $x\in \partial K$ is in at least one $B_k$. Note also that if $x\in B_k$, then $x + k \in B_{-k}$, so there are at least two $k$ such that $B_k$ is non-empty.

Since $K$ is bounded, only finitely many of the $B_k$ are non-empty. It is easy to see that each $B_k$ is closed as well. Since $\partial K$ is connected, it cannot be written as the finite union of two or more disjoint, non-empty closed sets. Therefore, we have that there exists $x \in \partial K, k_1 \not= k_2  \in \Z^n$ such that $x \in B_{k_1} \cap B_{k_2}$. The result then follows from Lemma \ref{nointersect}. 
\end{proof}

\begin{theorem}\label{thm:det2}
Let $A$ be an expansive, integer valued $n\times n$ matrix with determinant 2. Then 
\begin{enumerate}
\item there is no compact, centrally symmetric set $K$ which tiles $\R^n$ by integer translations such that $AK^\circ \supset K$.
\item there is no compact, convex set $K$ which tiles $\R^n$ by translations such that $AK^\circ \supset K$, and
\item there is no compact set $K$ whose boundary is connected and which tiles $\R^n$ by translations such that $AK^\circ \supset K$.
\end{enumerate}
\end{theorem}

\begin{proof}
Let $K$ be a compact, symmetric set which tiles $\R^n$ by $\Z^n$ translations. 
Note that $AK$ tiles $\R^n$ by $A\Z^n$ translations, and $A\Z^n$ is index 2 in $\Z^n$, so by Theorem \ref{index2}, $AK^\circ$ does not contain $K$. The other statements follow similarly.
\end{proof}

We end this section by stating and proving the promised improvement to \cite{BowRzeSpe2001, GuHan2000}.

\begin{lemma}\label{lemma:closedsets}
  Let $\tilde G$ be a compact subset of $\R^n$ such that $m(\partial \tilde G) = 0$. There exists a compact subset $G$ of $\tilde G$ with $m(\partial G) = 0$ such that $G$ packs $\R^n$ by integer translations and
 \[
 \bigcup_{k \in \Z^n} (G + k) = \bigcup_{k \in \Z^n} (\tilde G + k) \qquad a.e.
 \]
\end{lemma}

\begin{proof}
    Let $\{k_j\}_{j = 1}^\infty$ denote an enumeration of $\Z^n$. Let $T = [-1/2, 1/2]^n$. We define a collection of closed sets as follows. Let $C_1 = (T + k_1) \cap \tilde G$ and see that $m(\partial C_1) = 0$ and $C_1$ is closed.

    For $m \ge 2$, we define
    \[
    C_m = \tilde G \cap \left(T + k_m\right) \cap \overline{\left(\bigcup_{i = 1}^{m - 1}(C_i - k_i + k_m)\right)^C}
    \]

    The collection $\{C_i\}$ satisfies the following easily verified conditions
    \begin{enumerate}
        \item for every $i$, $C_i \subset T + k_i$,
        \item for every $i \not= j$, $m\left((C_i - k_i) \cap (C_j - k_j)\right) = 0$,
        \item for only finitely many $i$ is $C_i \not= \emptyset$,
        \item for every subset of $\tilde G$ of positive measure, there exists a further subset $H$ of positive measure, $k \in \Z^n$, and $i$ such that $H  + k \subset C_i$, and 
        \item $m(\partial C_i) = 0$ for every $i$.
    \end{enumerate}

    Let $G = \bigcup_{i = 1}^\infty C_i$. Items (1) and (2) together imply that $G$ packs $\R^n$ by integer translations. Item (3) and the fact that $C_i$ is closed for each $i$ implies that $G$ is compact. Item (4) implies that 
    \[
    \bigcup_{k \in \Z^n} (G + k) = \bigcup_{k \in \Z^n} (\tilde G + k) \qquad a.e.,
    \]
    and item (5) implies $m(\partial G) = 0$. 
\end{proof}

%Stopped here May 14

\begin{proposition}\label{prop:compactmra}
    Let $A$ be an expansive, integer matrix. There exists a compact set $K$ with $0 \in K^\circ$ which tiles $\R^n$ via integer translations such that $AK \supset K$. 
\end{proposition}

\begin{proof}

    Let $Q \subset [-1/2, 1/2]^n$ be an ellipsoid such that $AQ^\circ \supset Q$. Let $E_1 = Q$. 

    \[
    \tilde E_m = \overline{AE_{m - 1} \setminus \left(\bigcup_{i = 1}^{m - 1}\bigcup_{k \in \Z^n} \left(E_{m - 1} + k\right)\right)}.
    \]
    Note that there are only finitely many $k \in \Z^n$ for which $(E_{m - 1} + k) \cap AE_{m-1} \not= \emptyset$. Let $E_m$ be the subset of $\tilde E_m$ guaranteed to exist by Lemma \ref{lemma:closedsets}. Let $K = \cup_{i = 1}^\infty E_i$. 
    
    It is clear that $AK \supset K$ and that $0 \in K^\circ.$ It also follows from the definition of $\tilde E_m$ and Lemma \ref{lemma:closedsets} that $K$ packs $\R^n$ via translations. It remains to show that $K$ is compact and covers $\R^n$ by integer translations. We will do so by showing by induction on $J$ that for almost all $x \in \R^n$ if $J$ is the smallest nonnegative integer such that $A^{-J}x \in Q$, then $x + k \in \cup_{i = 1}^JE_i$ for some $k \in \Z^n$. This implies, in particular, that $K = \cup_{i = 1}^m E_i$ for some $m$, so $K$ is compact.

    The case $J = 1$ is clear.  Assume true for $J - 1 \ge 1$. For $x$ where $J$ is the smallest integer such that $A^{-j}x \in Q$, $y = A^{-1}x$ has $J - 1$ as the smallest integer for which $A^{-j}y \in Q$. For almost every such $y$, there exists $k\in \Z^n$ such that $y + k \in E_{m_0}$ for some $m_0 \le J - 1$. If $A(y + k) \in \tilde E_{m_0 + 1}$, then there exists $k_2$ such that $A(y + k) + k_2 \in E_{i + 1}$. Therefore, $x + Ak + k_2 \in E_{i + 1}$, and we are done.
    
    If $A(y + k) \not \in \tilde E_{m_0 + 1}$ then there exists $m_1 < m_0 + 1$ and $k_3$ such that $Ay + Ak + k_3 \in E_{m_1}$, which finishes the proof.
\end{proof}

\section{Two dimensional case}

The situation in two dimensions is considerably easier. In this section, we characterize the integer matrices which are strictly expansive with respect to a convex, centrally symmetric set. As mentioned in the proof of Theorem \ref{index2}, convex sets which tile $\R^n$ via translations are necessarily centrally symmetric; however, we will not omit the words ``centrally symmetric" in the sequel. We start by summarizing the results in \cite{BowSpe02MeyerType}. First, an integer matrix $A$ with trace $t$ and determinant $d$ is expansive if and only if $|t| \le d$ and $d \ge 2$, or $|t| \le -d - 2$ and $d \le -2$.

\begin{theorem}\label{2done}
    Let $A$ be an expansive, integer valued matrix. If $|\det(A)| > 2$ and $A$ is not similar to $\begin{pmatrix}
            0&1\\3&0
        \end{pmatrix}$ over $GL_2(\Z)$, then $A$ is strictly expansive.
\end{theorem}

A key step in the proof of Theorem \ref{2done} was finding convenient forms of expansive matrices, as in the following proposition.

\begin{theorem}\label{2dtwo}
        Let $A$ be an expansive, integer valued matrix. Then, $A$ is similar over $GL_2(\R)$ to a matrix $\begin{pmatrix}
            a&b\\c&d
        \end{pmatrix}$ of one of the following three forms:
        \begin{enumerate}
            \item $|b| \ge |a - d|$, $|c| \ge |a - d|$,
            \item $b = 0$ and $|c| \ge |a - d|$,
            \item $b = c = 0$
        \end{enumerate}
\end{theorem}

Implicit in the proof of Theorem \ref{2done} are the following two theorems, which give sufficient conditions for $A$ to be strictly expansive with respect to a centrally symmetric, convex set.

\begin{theorem}
  Let $A$ be an expansive, integer valued matrix with nonnegative trace which is not similar to $\begin{pmatrix}
      0&1\\-d&t
  \end{pmatrix}$ over $GL_2(\R)$. If $A$ is of type (i) or of type (iii), then $A$ is strictly expansive with respect to a convex, centrally symmetric set.

  If $A$ is of type (ii) and $a \not= d$, then $A$ is strictly expansive with respect to a convex, centrally symmetric set.
\end{theorem}

\begin{proposition}
      Let $A$ be an expansive, integer valued matrix with nonnegative trace which is similar to $\begin{pmatrix}
      0&1\\-d&t
  \end{pmatrix}$ over $GL_2(\R)$. 
  \begin{enumerate}
      \item If $d \ge 3$ and $t \ge 2$, then $A$ is strictly expansive with respect to a convex, centrally symmetric set.
      \item If $d \le -3$ and $0 \le t \le -d - 2$, then $A$ is strictly expansive with respect to a convex, centrally symmetric set.
  \end{enumerate}
\end{proposition}

In order to understand which $2 \times 2$ matrices are strictly expansive with respect to a convex, centrally symmetric set, we have only to consider the following cases.

\begin{enumerate}
    \item $A$ is similar to $\begin{pmatrix}
        a&b\\0&a
    \end{pmatrix}$.
    \item $A$ is similar to $\begin{pmatrix}
        0&1\\
        -d&t
    \end{pmatrix}$ for some $d \ge 3$ and $t = 1$.
    \item $A$ is similar to $\begin{pmatrix}
        0&1\\
        -d&t
    \end{pmatrix}$ for some $d \ge 3$ and $t = 0$.
    \item $|\det(A)| = 2$.
\end{enumerate}

Note that $|\det(A)| = 2$ is covered by Theorem \ref{thm:det2}. In the sequel, we will use the observation that if $x = (x_1, x_2)$ and $y = (y_1, y_2)$ are in a convex, centrally symmetric set $K$ and $x \not= \pm y$, then Gauss's area formula implies
\[
\frac{1}{2} |K| \ge |x_1 y_2 - x_2 y_1|.
\]
If $K$ is not the symmetric convex hull of $x$ and $y$, then the inequality is strict.

Finally, we will use the following lemma, which is Corollary 3.6 in \cite{BowSpe02} restated.

\begin{lemma}\label{cor36} Let $A = \begin{pmatrix}
    a&b\\c&d
\end{pmatrix}$ be an expansive, integer valued matrix. Let $K = \conv \{\pm \colvec{u/2}{1/2}, \pm \colvec{u/2 + 1}{1/2}$. Then $A(K^\circ) \supset K$ if and only if 
\begin{align}\label{maxu}
    \max \{&\left|-uc + a\right|, \left|-(u + 2)c + a\right|, \left|ud - b + (-1 - u)(-uc + a)\right| \nonumber \\ 
     &\left|(u + 2)d - b + (-1 - u)(-(u + 2)c + a)\right|\} < \left|\det(A)\right|
\end{align}
\end{lemma}

\begin{proposition}
    If $A$ is an expansive, integer valued matrix with nonnegative trace which is similar to $\begin{pmatrix}
        a&b\\0&a
    \end{pmatrix}$, then $A$ is strictly expansive with respect to a convex, centrally symmetric set if and only if $|b| < a^2$.
\end{proposition}

\begin{proof}
We claim that $u = \frac{\sgn(b) - a}{a}$ satisfies equation \eqref{maxu} condition in Lemma \ref{cor36} for the transpose of $A$ when $|b| < a^2$ and $a > 2$. Indeed, see that when $|b| = a^2$
    \[
    -ub + a = \sgn(b) a^2.
    \]
    Since $-u > 0$, it follows that $|-ub + a| < a^2$ whenever $|b| < a^2$. Similarly, when $|b| = a^2$, 
    \[
    -(u + 2)b + a = -\sgn(b) a^2.
    \]
    Since $-(u + 2) < 0$, it follows that $|-(u + 2)b + a| < a^2$ whenever $|b| < a^2$. Next, note when $a > 2$, since $|(u^2 + u)b| < a + 1$
    \begin{align*}
        |ua + (-1 - u)(-ub + a)| &= |-a + u^2b + ub|\\
        &< 2a + 1 < a^2.
    \end{align*}
    Finally, since $|2b(1 + u)| < 2$,
    \begin{align*}
        |(u + 2) a + (-1 - u)(-(u + 2)b + a)| &= |a + u^2b + ub + 2b + 2bu|\\
        &< 2a + 3 < a^2.
    \end{align*}
    When $a = 2$, $u = -1.4$ works when $b < 0$ and $u = -.58$ works when $b > 0$. Since the transpose of $A$ is similar to $A$ over $GL_n(\Z)$, we are done.
    
    Next, we consider the case when $|b| \ge a^2$. Let $K$ be a convex, centrally symmetric set that tiles $\R^2$ by translations. We start by showing that if $AK^\circ \supset K$, then  $\colvec{1/2}{0} \in K$. There exists $k_1, k_2 \in \Z$ such that $\colvec{1/2}{0} + \colvec{k_1}{k_2} \in K$.  If $k_2 = 0$, then $k_1 \in \{0, -1\}$ and by symmetry, $\colvec{1/2}{0} \in K$. Now, if $k_2 \not= 0$, then the points $w_1 = \colvec{1/2 + k_1}{k_2}$ and $w_2 = A^{-1} \colvec{1/2 + k_1}{k_2}$ are in $K$, and the symmetric convex hull of $\{w_1, w_2\}$ is contained in $K$. However, the measure of the symmetric convex hull of $\{w_1, w_2\}$ is $2k_2^2 |\frac{b}{a}|$, and since the measure of $K$ is less than or equal to 1, this means that $k_2 = 0$.

    Continuing the proof, since $\colvec{1/2}{0} \in K$, if $\colvec{x}{1/2 + k} \in K$ then $|K| \ge \left|\frac 12 + k\right|$, which implies that $k$ is either $0$ or $-1$, so $K \subset \R \times [-1/2, 1/2]$. 
    
    The convex subsets of $\R \times [-1/2, 1/2]$ which tile $\R^2$ by integer translations are exactly those considered in Lemma \ref{cor36}, so it suffices to check equation \eqref{maxu} for $A$. To that end, we note that the maximum reduces to
    \[
    \max\{|a|, |b \pm a|\},
    \]
     which is always larger than $a^2$ when $|b| \ge a^2$ and $A$ is expansive.
\end{proof}

 We present the following lemma without proof.

\begin{lemma}\label{lemma4}
    Let $K$ be a compact, convex, centrally symmetric set. If there exists $z$ such that $|z| \ge 1/2$ and any one of the following conditions hold:
    \begin{enumerate}
        \item $\colvec{0}{z} \in K^\circ$,
        \item $\colvec{z}{0} \in K^\circ$,
        \item $\colvec{x}{z} \in K$ and $\colvec{y}{z} \in K^\circ$, where the sign of $x$ and $y$ are different, or
        \item $\colvec{z}{x} \in K$ and $\colvec{z}{y} \in K^\circ$, where the sign of $x$ and $y$ are different,
    \end{enumerate}
     then $K$ does not tile $\R^2$ by translations.
\end{lemma}

\begin{proposition}
Let $A$ be an expansive, integer valued matrix that is similar to  $\begin{pmatrix} 0&1\\-d&t\end{pmatrix}$, where $d \ge 3$ and $t \in \{0, 1\}$. If $K$ is a compact, convex, centrally symmetric set that tiles $\R^2$ by translations, then $AK^\circ \not\supset K$.
\end{proposition}

\begin{proof} Suppose for the sake of contradiction that $K$ as in the example statement exists. Then for every $x \in K$, $A^{-1} x \in  K^\circ$, where $A^{-1} = \begin{pmatrix}
    t/d&-1/d\\1&0
\end{pmatrix}$. The goal of the proof is to show that there is no $k_1, k_2 \in \Z$ such that $\colvec{1/2 + k_1}{1/2 + k_2} \in K$.

We begin by showing that $\colvec{0}{1/2} \in K$. Note that for $y\in \R$ and $k \in \Z$, $A^{-1} \colvec{k}{y} = \colvec{(kt - y)/d}{k} \in K^\circ$, so if $\colvec{k}{y} \in K$, then 
\[
\frac{1}{2} |K| > k^2 + y^2/d - ytk/d.
\]
The minimum value $k^2 + y^2/d - ytk/d$ over $y\in \R$ is $k^2 (1 - \frac{1}{4d})$ when $t = 1$ and $k^2$ when $t = 0$. In either case, $k = 0$ is necessary for $|K| \le 1$. Therefore, if $\colvec{k_1}{1/2 + k_2} \in K$, $k_1 = 0$ and therefore $k_2 \in \{0, -1\}$, so $\colvec{0}{1/2} \in K$.

Next, we show that $\colvec{1/2}{\pm 1/2} \not \in K$. In the case $t = 1$, $A^{-1} \colvec{1/2}{1/2} = \colvec{0}{1/2} \in K^\circ$, so $K$ cannot tile by translations by Lemma \ref{lemma4}.  On the other hand, $A^{-1} \colvec{1/2}{-1/2} =\colvec{1/d}{1/2}$. So, if $\colvec{1/2}{-1/2}$ were in $K$, then $\colvec{-1/2}{1/2} \in K$ and $\colvec{1/d}{1/2} \in K^\circ$, so $K$ cannot tile by translations by Lemma \ref{lemma4}. 

When $t = 0$, $A^{-1} \colvec{1/2}{1/2} = \colvec{-1/(2d)}{1/2}$ and $A^{-1} \colvec{1/2}{-1/2}  = \colvec{1/(2d)}{1/2}$, each of which implies $K$ cannot tile by translations by Lemma \ref{lemma4} if $\colvec{1/2}{\pm 1/2} \in K$.

The third step is to show that $\colvec{1/2 + k_1}{1/2 + k_2}$ cannot be in $K$ for any choice of $k_1, k_2 \in \Z$. Suppose it were in $K$. Since $\colvec{0}{1/2} \in K$, 
\[
\frac{1}{2} |K| \ge 1/2 \left |1/2 + k_1\right |
\]
so $k_1 = 0$ or $k_1 = -1$. Since $K$ is symmetric, we can assume $k_1 = 0$.  The cases $k_2 = 0$ and $k_2 = -1$ were considered above. Note that $A^{-1} \colvec{1/2}{1/2 + k_2} = \colvec{(t - 1 - 2k_2)/2d}{1/2} \in K^\circ$.  If $k_2 \ge 1$, $t - 1 - 2k_2 < 0$, which together with $\colvec{1/2}{1/2 + k_2} \in K$ implies $\colvec{0}{1/2} \in K^\circ$. Lemma \ref{lemma4} then implies $K$ does not tile by translations. 

Finally, if $k_2 \le -2$, then by symmetry, $\colvec{-1/2}{-1/2 - k_2} \in K$, where $-1/2 - k_2 \ge 3/2$. Since $A^{-1} \colvec{1/2}{1/2 + k_2} = \colvec{(t - 1 - 2k_2)/2d}{1/2} \in K^\circ$, $K$ does not tile by translations by Lemma \ref{lemma4}.
\end{proof}

\bibliographystyle{plain}
\bibliography{bibliography.bib}

\end{document}